\documentclass[12pt]{amsart}
\usepackage{verbatim}
\usepackage{amscd}

\numberwithin{equation}{section}
\input epsf

\begin{document}

\renewcommand{\theequation}{\thesection.\arabic{equation}}
\setcounter{secnumdepth}{2}
\newtheorem{theorem}{Theorem}[section]
\newtheorem{definition}[theorem]{Definition}
\newtheorem{lemma}[theorem]{Lemma}
\newtheorem{corollary}[theorem]{Corollary}
\newtheorem{proposition}[theorem]{Proposition}
\numberwithin{equation}{section}
\theoremstyle{definition}
\newtheorem{example}[theorem]{Example}
\title[Toric GK structures. III]
{Toric generalized K$\ddot{A}$hler structures. III}

\author[Yicao Wang]{Yicao Wang}
\address
{Department of Mathematics, Hohai University, Nanjing 210098, China}
\maketitle

\baselineskip= 20pt
\begin{abstract}The paper clarifies some subtle points surrounding the definition of scalar curvature in generalized K$\ddot{a}$hler (GK) geometry. We have solved an open problem in GK geometry of symplectic type posed by R. Goto \cite{Go1} on relating the scalar curvature defined in terms of generalized pure spinors \emph{directly} to the underlying biHermitian structure. In particular, we apply this solution to toric GK geometry of symplectic type and prove that the scalar curvature suggested in this setting by L. Boulanger \cite{Bou} coincides with Goto's definition.
\end{abstract}
\section{Introduction}

The notion of scalar curvature in K$\ddot{a}$hler geometry is of special importance, especially in search of extremal metrics such as K$\ddot{a}$hler-Einstein metrics and cscK metrics. The scalar curvature is associated with the Levi-Civita connection. However, on a \emph{general} Hermitian manifold, there are several connections that could be considered, the Levi-Civita connection, the Chern connection, the Bismut connection and so on. To what extent one of these is important depends on the geometric problem one is considering. What makes K$\ddot{a}$hler geometry rather special is that all these connections actually coincide. This property would be surely lost when one starts to study non-K$\ddot{a}$hler geometry.

As a generalization of K$\ddot{a}$hler geometry, GK geometry still lacks a suitable notion of scalar curvature in its full generality. At the very early stage of GK geometry, it was fairly clear that the Bismut connection, rather the Levi-Civita one, should be more reasonable to be used. However, there are two copies of such a connection, $\nabla^+$ and $\nabla^-$, whose torsions differ from each other by a minus sign. These two connections are both necessary to treat the two underlying complex structures $J_\pm$ on the same footing. It is not clear what one has to do to combine them properly to give a notion of scalar curvature. Even we don't know whether these connections are enough to do so.

Recently, there are attempts to define the notion of scalar curvature in GK geometry. In \cite{Bou}, L. Boulanger considered toric GK structures of symplectic type and defined a notion of scalar curvature for these structures using an analogous version of the philosophy in K$\ddot{a}$hler geometry that scalar curvature should be the moment map in an infinite-dimensional GIT theory to find cscK metrics. In \cite{Go1}, R. Goto investigated general GK structures of sympletic type and succeeded in finding an analogue of Ricci form in terms of local data from \emph{generalized pure spinors} defining the geometry. Goto was able to extract a scalar curvature out of the Ricci form and prove in general that his version of scalar curvature is the moment map of the above-mentioned infinite-dimensional GIT theory. Later in \cite{Go2}, the way how the scalar curvature arises was interpreted in terms of \emph{generalized connection} in the sense of \cite{Gu2}. However, there is a gap between Boulanger's and Goto's works. In the theory of toric GK geometry of symplectic type \cite{Bou, Wang1, Wang2}, the initial data are essentially the biHermitian data underlying the geometry rather than generalized pure spinors in Goto's work, and these two kinds of data relate to each other in a highly non-linear way. Even the seemingly different definitions of scalar curvature cannot be directly compared. This made Goto to pose the problem of finding an expression of his scalar curvature in terms of the biHermitian data directly.

To solve Goto's problem is one of our goals in this paper. Our approach to the solution is to replace generalized pure spinors by their \emph{bi-spinor} cousins. If the GK manifold under consideration is spin and $S$ is the spinor bundle, it is well-known that the bi-spinor bundle $S\otimes S$ is isomorphic to the bundle $\mathcal{S}$ of generalized spinors, in other words, the bundle of forms. In particular, the presence of two complex structures $J_\pm$ provides two canonical spinor bundles $S_1$ and $S_2$. We can replace $S\otimes S$ with $S_1\otimes S_2$ and consequently the latter can on one side be directly connected to $J_\pm$ and on the other side be connected to the generalized pure spinors defining the GK geometry. This bi-spinor approach towards (generalized) Riemannian geometry is not new and is often used in compactification of string theory with non-trivial flux. See \cite{Wit} for example. What's new here is that we must specify the construction of spinor bundles using $J_\pm$ and clarify that all relevant structures expressed using generalized spinors can be transported to the bi-spinor side.

The second goal of this paper is to apply the solution to toric GK manifolds of symplectic type and establish the equivalence of Boulanger's and Goto's definitions of scalar curvature in this setting. This equivalence can only be justified after we have refined Goto's definition of Ricci form. Goto's investigation depends heavily on generalized pure spinors and contains some subtleties to be clarified. We show that the Ricci curvature can be totally interpreted in terms of ordinary connections. This refinement makes it \emph{conceptually} clear how to analyze the toric case in which there is a canonical way to find the proper bi-spinors.

The paper is organized as follows. In \S~\ref{sec1}, we collect some background material on GK geometry. \S~\ref{sec2} is a short section covering the biholomorphic geometry on a GK manifold. This is crucial to motivate our bi-spinor approach. As a byproduct we suggest a notion of scalar curvature in this section, which however will not be used later. \S~\ref{sec3} is devoted to establishing the precise relation between bi-spinors and generalized spinors. The point is to show that all relevant structures involved in Goto's problem can be transported to the bi-spinor side. In \S~\ref{sym}, we recall the basics of GK structures of symplectic type and the construction in \cite{Wang1, Wang2} which provides plenty of toric examples. In \S~\ref{goto}, Goto's definition of Ricci form and scalar curvature is reviewed and refined in terms of ordinary connections. In particular, several subtle points surrounding the Ricci curvature is clarified. The last section \S~\ref{sec5} is to address the scalar curvature of a toric GK manifold of symplectic type. In \S~\ref{bou}, Boulanger's formalism to define the scalar curvature is recalled and as a corollary we derive the general formula of scalar curvature which was missing in \cite{Bou}. \S~\ref{goto2} is devoted to deriving an explicit expression of Goto's scalar curvature in the setting of toric GK geometry of symplectic type. This involves a detailed analysis of the generalized holomorphic structure on the canonical line bundle of (generalized) pure spinors. The final expression of the scalar curvature is given in Prop.~\ref{scal}. In \S~\ref{equ}, this expression is shown to be equivalent to the one derived in Boulanger's formalism.
\section{Basics of GK geometry}\label{sec1}
In this section, we collect the most relevant knowledge on GK geometry. The basic references are \cite{Gu00, Gu0, Gu2}.

Generalized geometry starts with the basic idea of replacing tangent bundles by exact Courant algebroids. A Courant algebroid $E$ is a real vector bundle $E$ over a smooth manifold $M$, together with an anchor map $\pi$ to the tangent bundle $T$ of $M$, a non-degenerate inner product $(\cdot, \cdot)$ and a so-called Courant bracket $[\cdot , \cdot]_c$ on $\Gamma(E)$. These structures should satisfy some compatibility axioms we won't recall here. $E$ is called exact, if the short sequence \[0\longrightarrow T^*\stackrel{\pi^*}\longrightarrow E \stackrel{\pi}\longrightarrow T \longrightarrow0\]
is exact. We only consider exact Courant algrbroids in this paper. Given $E$, one can always find an isotropic right splitting $\mathfrak{s}:T\rightarrow E$, with a curvature form $H\in \Omega_{cl}^3(M)$ defined by
\[H(X,Y,Z)=([\mathfrak{s}(X),\mathfrak{s}(Y)]_c,\mathfrak{s}(Z)),\quad X, Y, Z\in \Gamma(T).\]
  By the bundle isomorphism $\mathfrak{s}+\pi^*:T\oplus T^*\rightarrow E$, the Courant algebroid structure on $E$ can be transported onto $T\oplus T^*$. Then the inner product $(\cdot,\cdot)$ is the natural pairing, i.e.
$( X+\xi,Y+\eta)=\xi(Y)+\eta(X)$, and the Courant bracket is
\begin{equation*}[X+\xi, Y+\eta]_H=[X,Y]+\mathcal{L}_X\eta-\iota_Yd\xi+\iota_Y\iota_XH,\end{equation*}
called the $H$-twisted Courant bracket. Different splittings are related by B-tranforms: \[e^B(X+\xi)=X+\xi+B(X),\]
where $B$ is a real 2-form.
\begin{definition}
 A generalized complex (GC) structure on a Courant algebroid $E$ is a complex structure $\mathbb{J}$ on $E$ orthogonal w.r.t. the inner product and its $\sqrt{-1}$-eigenbundle $L\subset E\otimes\mathbb{C}$ is involutive under the Courant bracket.
\end{definition}
For a GC structure $\mathbb{J}$, we have the decomposition $E\otimes{\mathbb{C}}=L\oplus \bar{L}$. Since $\mathbb{J}$ and its $\sqrt{-1}$-eigenbundle $L$ are equivalent notions, we shall occasionally use them interchangeably to denote a GC structure. For $H\equiv0$, ordinary complex and symplectic structures are extreme examples of GC structures. Precisely, for a complex structure $I$ and a symplectic structure $\omega$, the corresponding GC structures are of the following form:
\[\mathbb{J}_I=\left(
                 \begin{array}{cc}
                   -I & 0 \\
                   0 & I^* \\
                 \end{array}
               \right),\quad \mathbb{J}_\omega=\left(
                                                 \begin{array}{cc}
                                                   0 & \omega^{-1} \\
                                                   -\omega & 0 \\
                                                 \end{array}
                                               \right).
\]

A GC structure $L$ is an example of complex Lie algebroids. Via the inner product, $\wedge^\cdot L^*$ can be identified with $\wedge^\cdot \bar{L}$, and we then have an elliptic differential complex $(\Gamma(\wedge^\cdot \bar{L}), d_L)$, which induces the Lie algebroid cohomology associated with the Lie algebroid $L$. The differential complex can be twisted by an $L$-module.
\begin{definition}
Given a GC structure $L$ over $M$, an
$L$-connection $\mathrm{D}$ in a complex vector bundle $W$ is a linear differential operator
$\mathrm{D}:\Gamma(W)\longrightarrow \Gamma(\bar{L}\otimes W)$
satisfying
$$\mathrm{D}(f s)=d_Lf\otimes s+f\mathrm{D}s,\quad s\in \Gamma(W),\ f\in C^\infty(M).$$
If $\mathrm{D}$ is flat, i.e. $\mathrm{D}^2=0$, $\mathrm{D}$ is called a generalized holomorphic (GH) structure and $W$ an $L$-module or a GH vector bundle.
\end{definition}

There is also a notion of generalized connection in generalized geometry \cite{Gu2}.
\begin{definition}
A generalized connection in a vector bundle $W$ over $M$ is a linear differential operator $D:\Gamma(W)\rightarrow \Gamma(E\otimes W)$
such that 
\[D(fs)=\pi^*df\otimes s+fDs.\quad s\in \Gamma(W),\ f\in C^\infty(M).\]
 If further $W$ has an Hermitian structure $h$, then $D$ is compatible with $h$ when $D$ preserves $h$ in the usual sense.
\end{definition}
 Given a splitting, a generalized connection $D$ is the summation of an ordinary connection $\nabla$ and an endomorphism-valued vector filed $\chi$. The latter is canonical while the former will change if a different splitting is chosen. If $W$ is further equipped with a GH structure $\mathrm{D}$ w.r.t. a GC structure $\mathbb{J}$, then there is a canonical generalized connection $D$ which is compatible with the Hermitian metric and whose $\bar{L}$-part is precisely $\mathrm{D}$. This is the generalized analogue of the fact that a holomorphic structure and an Hermitian metric on a complex vector bundle uniquely determine the Chern connection.

 The following is the generalized version of a Riemannian metric.

\begin{definition}A generalized metric on a Courant algebroid $E$ is an orthogonal, self-adjoint operator $\mathcal{G}$ such that $( \mathcal{G}\cdot,\cdot)$ is positive-definite on $E$. It is necessary that $\mathcal{G}^2=\textup{Id}$.
 \end{definition}
 A generalized metric induces a \emph{canonical} isotropic splitting: $E=\mathcal{G}(T^*)\oplus T^*$. It is called \emph{the metric splitting}. Given a generalized metric, we shall often choose its metric splitting to identify $E$ with $T\oplus T^*$. However, other splittings are also used in specific situations in this paper later. Note that in the metric splitting $\mathcal{G}$ is of the form $\left(\begin{array}{cc} 0 & g^{-1} \\g & 0 \\
\end{array} \right)$ with $g$ an ordinary Riemannian metric.

  A generalized metric is an ingredient of a GK structure.
\begin{definition}
A GK structure on $E$ is a pair of commuting GC structures $(\mathbb{J}_1,\mathbb{J}_2)$ such that $\mathcal{G}=-\mathbb{J}_1 \mathbb{J}_2$ is a generalized metric.
\end{definition}
Since $\mathbb{J}_1, \mathbb{J}_2$ commute, $(T\oplus T^*)\otimes\mathbb{C}$ can be decomposed further:
$$(T\oplus T^*)\otimes\mathbb{C}=L_+\oplus L_-\oplus \bar{L}_+\oplus \bar{L}_-,$$
where $L_\pm$ are the $(\sqrt{-1},\pm \sqrt{-1})$-eigenbundles of
$(\mathbb{J}_1,\mathbb{J}_2)$ respectively; in particular, the $\sqrt{-1}$-eigenbundle of $\mathbb{J}_1$ is thus $L_1=L_+\oplus L_-$.

A GK structure can also be reformulated in a biHermitian manner: There are two complex structures $J_\pm$ on $M$ compatible with the metric $g$ induced from the generalized metric. Let $\omega_\pm=gJ_\pm$. Then \emph{in the metric splitting} the GC structures and the corresponding biHermitian data are related by the Gualtieri map:
 \[\mathbb{J}_1=\frac{1}{2}\left(
  \begin{array}{cc}
    -J_+-J_-& \omega_+^{-1}-\omega_-^{-1} \\
    -\omega_++\omega_- & J_+^*+J_-^* \\
  \end{array}
\right),\quad \mathbb{J}_2=\frac{1}{2}\left(
  \begin{array}{cc}
    -J_++J_-& \omega_+^{-1}+\omega_-^{-1} \\
    -\omega_+-\omega_- & J_+^*-J_-^* \\
  \end{array}
\right),\]
and in particular
\[L_\pm=(\textup{Id}\pm g)T_\pm^{0,1},\]
where $T_\pm^{0,1}$ are the $(0,1)$-part of $T\otimes\mathbb{C}$ w.r.t. $J_\pm$ respectively. Integrability of $(\mathbb{J}_1, \mathbb{J}_2)$ then takes the following form:
\[d^c_+\omega_++d^c_-\omega_-=0,\quad dd^c_+\omega_+=0,\]
where $d^c_\pm=[d, J^*_\pm]$. Then $H=-d^c_+\omega_+$ and we have two Bismut connections $\nabla^{\pm}$ on $T$:
\[\nabla^\pm_XY=\nabla_X^{LC}Y\pm\frac{1}{2}g^{-1}H(X, Y),\]
where $\nabla^{LC}$ is the Levi-Civita connection and we have viewed $H$ as a map from $T\times T$ to $T^*$. An important property of $\nabla^\pm$ is that they preserve $J_\pm$ respectively, just as on a K$\ddot{a}$hler manifold the underlying complex structure is preserved by the Levi-Civita connection.

We are particularly interested in GH line bundles on a GK manifold. In this setting, we always choose $L_1$ to be the underlying GC structure. Due to the decomposition $L_1=L_+\oplus L_-$, a GH structure $\mathrm{D}$ can be decomposed as $\mathrm{D}=\bar{\delta}_++\bar{\delta}_-$ accordingly. Actually, $\bar{\delta}_\pm$ are, in essence, ordinary $J_\pm$-holomorphic structures respectively. Additionally, it is necessary that \begin{equation}\bar{\delta}_+\bar{\delta}_-+\bar{\delta}_-\bar{\delta}_+=0.\label{ghc}\end{equation}
 Conversely, given $J_\pm$-holomorphic structures $\bar{\delta}_\pm$, if Eq.~(\ref{ghc}) is also satisfied, then $\mathrm{D}:=\bar{\delta}_++\bar{\delta}_-$ is a GH structure. If the GH line bundle is additionally endowed with an Hermitian metric, then the canonical generalized connection takes a simple form in the metric splitting:
 \[D=\frac{1}{2}(\nabla_++\nabla_-)+\frac{1}{2}g^{-1}(\nabla_+-\nabla_-),\]
 where $\nabla_\pm$ are the Chern connections associated to $\bar{\delta}_\pm$ respectively.

\section{Biholomorphic geometry on GK manifolds}\label{sec2}
This short section is mainly devoted to emphasizing a few simple yet key observations to motivate our later considerations. A more thorough account of the material presented here can be found in \cite{Gu3}.

Let $(M, J_\pm, g)$ be the biHermitian data associated to a GK manifold and $R^\pm$ be the curvature of $\nabla^{\pm}$ respectively. By abusing of notation, we also regard $R^\pm$ as covariant tensors, i.e., for $X, Y, Z, W\in T$,
\begin{equation}R^\pm(X,Y,Z,W)=g(R^\pm(X,Y)Z, W).\label{Bismut}\end{equation}
Then the first basic identity is
\begin{equation}R^+(X,Y,Z,W)=R^-(Z,W, X,Y).\label{B1}\end{equation}
To the author's knowledge, this fact goes back as early as to Bismut \cite{Bis}. The next basic fact seems never explicitly mentioned in the published literature:
\begin{lemma} \label{key}For $X,Y,Z,W\in T$,
\begin{equation}R^+(J_-X,J_-Y,Z,W)=R^+(X,Y,Z,W),\label{basic1}\end{equation}
\begin{equation}R^-(J_+X,J_+Y,Z,W)=R^-(X,Y,Z,W).\label{basic2}\end{equation}
\end{lemma}
\begin{proof}
Note that
\begin{eqnarray*}R^+(J_-X,J_-Y,Z,W)&=&R^-(Z, W, J_-X,J_-Y)\\
&=&g(\nabla^-_Z\nabla_W^-(J_-X)-\nabla^-_W\nabla_Z^-(J_-X)-\nabla^-_{[Z,W]}(J_-X),J_-Y)\\
&=&g(J_-R^-(Z,W)X,J_-Y)=R^-(Z,W,X,Y)\\
&=& R^+(X,Y,Z,W),\end{eqnarray*}
where we have used Eq.~(\ref{Bismut}), Eq.~(\ref{B1}) and the facts that $\nabla^-$ preserves $J_-$ and $g$ is Hermitian w.r.t. $J_-$. The second identity can be similarly proved.
\end{proof}
Before proceeding further, let us give some comments on the implication of the above lemma to motivate our later development. Let $T_{\pm}^{1,0}$ be the $J_\pm$-holomorphic tangent bundles respectively. Since $\nabla^+$ preserves $J_+$, $\nabla^+$ can be restricted on $T^{1,0}_+$ and it is compatible with the Hermitian metric $h_+$ on $T^{1,0}_+$ induced from $g$. Then Eq.~(\ref{basic1}) precisely means that $T^{1,0}_+$ is a complex vector bundle equipped with an Hermitian connection whose curvature is of type (1,1) w.r.t. the complex structure $J_-$. Consequently, $T^{1,0}_+$ is a holomorphic vector bundle over the complex manifold $(M, J_-)$ and $\nabla^+$ is actually the Chern connection associated to this holomorphic structure and the Hermitian metric $h_+$. A similar argument applies to $T^{1,0}_-$, using Eq.~(\ref{basic2}). We summarize the above observation in the following proposition.
\begin{proposition}Let $h_\pm$ be the Hermitian metrics on $T^{1,0}_\pm$ induced from $g$. The connection $\nabla^\pm$ gives rise to a $J_\mp$-holomorphic structure on $T^{1,0}_\pm$ and it is actually the Chern connection associated to this holomorphic structure and the Hermitian metric $h_\pm$.
\end{proposition}
This simple observation is responsible for many phenomena in GK geometry. For example, let $K_\pm$ be the canonical line bundle of $J_\pm$ respectively. Then the line bundles $K_\pm$, $K_+\otimes K_-$ and $K_+^{-1}\otimes K_-^{-1}$ are all both $J_+$-holomorphic and $J_-$-holomorphic in a natural way. Actually, more is true:
\begin{proposition}\label{holo2}With the above biholomorphic structure, $K_+\otimes K_-$ is actually a $\mathbb{J}_1$-GH line bundle. The claim holds as well for $K_+^{-1}\otimes K_-^{-1}$.
\end{proposition}
\begin{proof}
We postpone a brief proof until \S~\ref{spGK}.
\end{proof}

As another application of Lemma~\ref{key}, we suggest a \emph{general} definition of scalar curvature in GK geometry. However, this definition will not be used afterward. Note that the anti-canonical line bundle $K_+^{-1}$ of $(M, J_+)$ is a $J_-$-holomorphic line bundle and we still denote the compatible connection by $\nabla^+$. Let $R$ be the curvature of this connection.
\begin{definition}Since $\sqrt{-1}R$ is of type (1,1) w.r.t. $J_-$, we can take the trace of $\sqrt{-1}R$ w.r.t. the Hermitian metric $h_-$ on $(M, J_-)$ induced from $g$, i.e.
\[\kappa_c:=2R_{i\bar{j}}h_-^{i\bar{j}}.\]
We call $\kappa_c$ the canonical scalar curvature of the GK manifold under consideration.
\end{definition}
\emph{Remark}. If $J_+=J_-$, we are in the ordinary K$\ddot{a}$hler case and the definition coincides with the usual one. We can certainly define a scalar curvature in a similar manner using $K_-^{-1}$, $\nabla^-$ and the Hermitian metric $h_-$. However, due to Eq.~(\ref{B1}) and Lemma~\ref{key}, this will give the same result. Therefore our definition above actually treats the two complex structures $J_\pm$ on the same footing.

\section{Ordinary spinors and generalized spinors in GK geometry}\label{sec3}
\subsection{Bi-spinors}\label{bi}Material in this subsection is, more or less, well-established in the literature. However, we find it hard to obtain standard references which fit in well with our present context. Thus we choose to collect the necessary material scattered in the literature and spell out some details. As for general spin geometry, \cite{Law} is the basic reference. 

Let $(V,g)$ be a $2n$-dimensional Euclidian vector space and $Cl(V,g)$ the associated Clifford algebra subject to the relation
\[v\cdot w+w\cdot v=2g(v, w).\] There is a canonical isomorphism $\mathfrak{J}$ from $Cl(V,g)$ to $\wedge^\cdot V$ as vector spaces (we
will also identify $V$ with $V^*$ using $g$). Indeed, if $e_1,\cdots,e_p\in V$ are orthogonal, then
$\mathfrak{J}(e_1\cdots e_p)=e_1\wedge\cdots\wedge e_p$. If $v\in V, a\in Cl(V,g)$, then
\begin{equation}\mathfrak{J}(v\cdot a)=v\wedge \mathfrak{J}(a)+\iota_v\mathfrak{J}(a),\quad\quad \mathfrak{J}(a\cdot v)=v\wedge\widetilde{\mathfrak{J}(a)}-\iota_v\widetilde{\mathfrak{J}(a)},\label{cw}\end{equation}
where $\iota$ denotes contraction via $g$, and
$\widetilde{\cdot}$ indicates multiplication by $\pm1$ depending on the parity of $a$.

 Up to isomorphism $Cl(V,g)\otimes\mathbb{C}$ has a unique irreducible complex
representation $(\Delta,\rho)$ such that
$$\rho: Cl(V,g)\otimes\mathbb{C}\cong \text{End}(\Delta).$$
As a $\textup{Spin}(2n)$-module, this representation $\Delta$ can be written as $\Delta_+\oplus\Delta_-$ according to chirality, i.e., the eigenvalues of $\gamma:=(\sqrt{-1})^ne_1e_2\cdots e_{2n}$ for an oriented orthonormal basis $\{e_i\}$. $\Delta_\pm$ are both of dimension $2^{n-1}$.

If $(V,g)$ is further equipped with a $g$-compatible complex structure $J$, the above essentially unique representation can be constructed explicitly. Actually, let $V_{0,1}^*$ be the $(0,1)$-part of $V^*\otimes\mathbb{C}$. Then we can use $\wedge^\cdot V_{0,1}^*$ as the space of spinors. If $\{f_i\}$ is a unitary basis of $V^{1,0}$, and $\{f^i\}$ its dual, then we can take the representation of $Cl(V,g)\otimes\mathbb{C}$ on $\wedge^\cdot V_{0,1}^*$ to be
\begin{equation}\label{cr}\rho(\overline{f}_i)=\sqrt{2}\iota_{\overline{f}_i},\quad \rho(f_i)=\sqrt{2}\overline{f}^i\wedge\cdot.\end{equation}
 In fact, some more structures are introduced by this specification: $\Delta$ has a natural $\mathbb{Z}$-grading and the Hermitian metric $g$ equips $\Delta$ with a natural metric $h$, with respect to which the representation of $\textup{Spin}(2n)$ is unitary and components of different degrees are orthogonal. In particular, we have
\[h(a\cdot \phi,\psi)=h(\phi, a^\dag\cdot \psi),\quad \phi, \psi\in \Delta, a\in Cl(V,g)\otimes\mathbb{C}.\]
where $\dag$ is the conjugate linear extension of the order-reversing map $\mu(e_1\cdots e_p)=e_p\cdots e_1$ for orthogonal vectors $e_1,\cdots, e_p$.

Let us move on to an Hermitian manifold $(M, g, J)$ of complex dimension $n$ and we then have the Clifford bundle $Cl(T, g)$. For simplicity, assume that the canonical line bundle $K$ of $J$ has a square root $K^{1/2}$ and therefore $M$ is spin (\emph{as for the bi-spinor construction this requirement can finally be dropped }). Then the spinor bundle is\footnote{The square root $K^{1/2}$ is inserted to guarantee that we can really lift the structure group of $M$ from $\textup{SO}(2n)$ to $\textup{Spin}(2n)$. }
\[S=\oplus_{p}(\wedge^pT^*_{0,1}\otimes K^{1/2})\]
with (we choose the orientation induced by $J$)
\[S_+=\oplus_{even}(\wedge^pT^*_{0,1}\otimes K^{1/2}),\quad S_-=\oplus_{odd}(\wedge^pT^*_{0,1}\otimes K^{1/2}).\]
In particular, $J$ is characterized by the pure spinor line bundle $K^{1/2}\subset S_+$ in the following way:
 \[T_{0,1}=\{X\in T\otimes \mathbb{C}|X\cdot \phi=0, \forall \phi\in K^{1/2}\}.\]
 There is a canonical $\textup{Spin}(2n)$-invariant bilinear pairing $q: S\times S\rightarrow M\times \mathbb{C}$. Precisely if locally $s_0$ is a unitary frame of $K^{1/2}$ w.r.t. the canonical metric, and $\phi=f_Id\bar{z}_I\otimes s_0, \psi=g_{I'}d\bar{z}_{I'}\otimes s_0\in \Gamma(S)$ where $I, I'$ are multi-indices, then
\[q(\phi, \psi)=(\mu(f_Id\bar{z}_I)\wedge g_{I'}d\bar{z}_{I'})_{top}/\overline{s_0^2},\]
where $\mu(d\bar{z}_1\cdots d\bar{z}_p)=d\bar{z}_p\cdots d\bar{z}_1$ and "top" means taking the component of top degree. $q$ gives rise to an isomorphism $q^\sharp$ from $S\otimes S$ to $\textup{Hom}(S)$ via
\[q^\sharp(\phi\otimes \psi)(\varphi)=q(\psi, \varphi)\phi,\quad \phi, \psi, \varphi\in \Gamma(S).\]
Followed by $\mathfrak{J}\circ\rho^{-1}$, this gives an isomorphism
\[\mathfrak{J}\circ\rho^{-1}\circ q^\sharp: S\otimes S\rightarrow\wedge^\cdot (T^*\otimes\mathbb{C}).\]
 For $\phi, \psi\in S$ we will call $\phi\otimes \psi$ a bi-spinor and use $[\phi\otimes \psi]$ to denote its image under the above isomorphism. Note that on one side $S\otimes S$ has its natural Hermitian metric induced from that on $S$ and on the other side $\wedge^\cdot (T^*\otimes\mathbb{C})$ also has its Hodge metric induced from $g$. What matters for us is the following
\begin{proposition}The Hodge metric on $\wedge^\cdot (T^*\otimes\mathbb{C})$ can be rescaled with a constant such that the map $\mathfrak{J}\circ\rho^{-1}\circ q^\sharp$ is an isometry from $S\otimes S$ to $\wedge^\cdot (T^*\otimes\mathbb{C})$.
\end{proposition}
\begin{proof}The metric on $\wedge^\cdot (T^*\otimes\mathbb{C})$ can be pulled back to $Cl(T,g)\otimes\mathbb{C}$ through $\mathfrak{J}$. Denote this metric by $h_R$. $h_R$ can be characterized in another manner: for $a, b\in Cl(T,g)\otimes\mathbb{C}$,
\[h_R(a,b)=\textup{tr}(a^\dag b),\]
where $\textup{tr}(a)$ means taking the degree-0 part of $\mathfrak{J}(a)$. By identifying $Cl(T,g)\otimes\mathbb{C}$ with $\textup{Hom}(S)$, we have another Hermitian metric $h_R'$ on $Cl(T,g)\otimes\mathbb{C}$: for $a, b\in Cl(T,g)\otimes\mathbb{C}$,
\[h_R'(a,b)=\textup{Tr}(a^\dag b),\]
where on the right side $a, b$ are interpreted as linear operators on $S$, $a^\dag$ is the usual conjugate operator of $a$ w.r.t. the metric on $S$ and $\textup{Tr}(a)$ is the usual trace of $a$ as a linear operator on $S$. It can be shown by induction on $n$ that
\[h_R'(a,b)=2^nh_R(a,b).\]
Then to prove our claim we only have to check that
\[h'_R(q^\sharp(\phi\otimes\psi),q^\sharp(\phi\otimes\psi))=h(\phi,\phi)\times h(\psi,\psi).\]

We can define a conjugate linear operator $c:S\rightarrow S$ by \footnote{"c" is often called the charge conjugate operator in the physical literature.}
\[q(\phi, \cdot)=h(\phi^c, \cdot).\]
By choosing a local unitary frame $\{\vartheta_j\}$ of $S$, one can check easily that $c$ preserves the metric $h$. Now for any $\varphi_1, \varphi_2\in S$
\begin{eqnarray*}h(q^\sharp(\phi\otimes\psi)(\varphi_1),\varphi_2)&=&\overline{q(\psi, \varphi_1)}h(\phi,\varphi_2)=\overline{h(\psi^c,\varphi_1)}h(\phi,\varphi_2)\\
&=&h(\varphi_1, \psi^c)h(\phi,\varphi_2)=h(\varphi_1,h(\phi,\varphi_2)\psi^c). \end{eqnarray*}
This means that the conjugate of $q^\sharp(\phi\otimes\psi)$ is given by $h(\phi,\cdot)\psi^c$. Consequently,
\begin{eqnarray*}h'_R(q^\sharp(\phi\otimes\psi),q^\sharp(\phi\otimes\psi))&=&\sum_jh(q^\sharp(\phi\otimes\psi)\vartheta_j, q^\sharp(\phi\otimes\psi)\vartheta_j)\\
&=&\sum_j h(\vartheta_j, \psi^c)h(\psi^c, \vartheta_j)h(\phi,\phi)\\
&=&h(\psi^c,\psi^c)h(\phi,\phi)\\
&=&h(\psi,\psi)h(\phi,\phi).
\end{eqnarray*}
This completes our proof.
\end{proof}
In generalized geometry, forms are themselves interpreted as spinors. Actually, equipped with its natural pairing, the generalized tangent bundle $T\oplus T^*$ is a quadratic vector bundle and thus has its corresponding Clifford bundle $Cl(T\oplus T^*)$. Typically, $X+\xi\in T\oplus T^*$ acts on a form $\Phi$ by
\[(X+\xi)\cdot \Phi=\iota_X\Phi+\xi\wedge \Phi.\]
If additionally $T$ is equipped with a Riemannian metric $g$, then $Cl(T\oplus
T^*)$ can be identified with the graded tensor product
$Cl(T,g)\otimes Cl(T,-g)$ via the extension of
$$v\otimes w\rightarrow v_+\cdot w_-,$$
where $\cdot$ denotes multiplication in $Cl(T\oplus T^*)$, and $v_\pm=v\pm g(v)$.
\begin{proposition}\label{bs1}For any $v\in T$, $\phi, \psi\in S$,
$$[v\cdot\phi\otimes\psi]=v_+\cdot[\phi\otimes\psi],\quad [\phi\otimes
v\cdot\psi]=-v_-\cdot\widetilde{[\phi\otimes\psi]}.$$
\end{proposition}
\begin{proof}Due to Eq.~(\ref{cw}),
\begin{eqnarray*}
[v\cdot\phi\otimes\psi]&=&\mathfrak{J}\circ\rho^{-1}(v\circ q^\sharp (\phi\otimes\psi))\\
&=&g(v)\wedge[\phi\otimes\psi]+\iota_v[\phi\otimes\psi]\\
&=&v_+\cdot[\phi\otimes\psi].
\end{eqnarray*}
Likewise,
\begin{eqnarray*}
[\phi\otimes
v\cdot\psi]&=&\mathfrak{J}\circ\rho^{-1}(q^\sharp(\phi\otimes\psi)\circ\mu(v))\\
&=&\mathfrak{J}\circ\rho^{-1}(q^\sharp(\phi\otimes\psi)\circ v)\\
&=&g(v)\wedge\widetilde{[\phi\otimes\psi]}-\iota_v\widetilde{[\phi\otimes\psi]}\\
&=&-v_-\cdot\widetilde{[\phi\otimes\psi]},
\end{eqnarray*}
where we have used the property that \[q(a\cdot \phi, \theta)=q(\phi, \mu(a)\cdot \theta),\quad a\in Cl(T, g), \phi, \theta\in S.\]
\end{proof}
\subsection{Spinors in GK geometry}\label{spGK}If $(M, \mathbb{J}_1, \mathbb{J}_2)$ is a GK manifold of real dimension $2n$, then we can use the underlying metric splitting to identify the Courant algebroid with $T\oplus T^*$
and $\mathcal{S}:=\wedge^\cdot (T^*\otimes\mathbb{C})$ can be viewed as a $T\oplus T^*$-spinor bundle. There is a $\wedge^{2n}T^*$-valued $\textup{Spin}_0 (2n,2n)$-invariant \footnote{$\textup{Spin}_0 (2n,2n)$ is the identity component of $\textup{Spin} (2n,2n)$.} bilinear form called \emph{Chevalley pairing}:
\[(\Phi, \Psi)_{Ch}=(\Phi\wedge \mu(\Psi))_{top}.\]
We give $M$ the orientation induced by $J_+$ and let $e_1,\cdots, e_{2n}$ be a positively oriented orthonormal local frame of $T$, then  \[\star:=e_{2n}^+\cdots e_2^+e_1^+\in \Gamma(Cl(T\oplus T^*))\] is globally well-defined, where $e_i^+=e_i+g(e_i)$. We can define an Hermitian metric on $\mathcal{S}$ by
\[(\Phi, \Psi)_hvol_g=(\Phi, \star\cdot \bar{\Psi})_{Ch},\]
which is nothing else but the usual Hodge metric on differential forms.

Since both $\mathbb{J}_1$ and $\mathbb{J}_2$ lie in $\mathfrak{so}(E)$, they can be
``quantized" to lie in $\mathfrak{spin}(E)$ and hence give
rise to a bigrading on $\mathcal{S}$
 according to distinct eigenvalues. Denote the
$(k\sqrt{-1},l\sqrt{-1})$-eigenbundle of $(\mathbb{J}_1,\mathbb{J}_2)$ by
$U_{k,l}$. Then we obtain a decomposition of
$\mathcal{S}$:
\[
\begin{array}{ccccccc}
\quad &\quad &\quad &U_{0,n} &\quad&\quad&\quad\\
\quad &\quad &\cdots &\quad&\cdots &\quad&\quad\\
\quad &U_{-n+1,1}&\quad&\quad&\quad& U_{n-1,1}&\quad\\
U_{-n,0}&\quad&\quad&\cdots&\quad&\quad& U_{n,0}\\
\quad &U_{-n+1,-1}&\quad&\quad&\quad& U_{n-1,-1}&\quad\\
\quad &\quad &\cdots &\quad&\cdots &\quad&\quad\\
\quad &\quad &\quad &U_{0,-n} &\quad&\quad&\quad
\end{array}.
\]
In particular, $U_{n,0}$ is a line bundle characterizing the $\sqrt{-1}$-eigenbulde of $\mathbb{J}_1$ in the following sense:
\[L_1=\{X+\xi\in (T\oplus T^*)\otimes \mathbb{C}|(X+\xi)\cdot \Phi=0, \forall \Phi\in U_{n,0}\}.\]
A more detailed account of the above material can be found in \cite{BCG}. A key result of \cite{BCG} is \footnote{It should be remarked that our convention is slightly different from that in \cite{BCG}. }
\begin{lemma}$\star$, when acting on $\mathcal{S}$, preserves the bigrading and in particular acts on $U_{k,l}$ by multiplying $(-1)^n\times(\sqrt{-1})^{k+l}$.
\end{lemma}

 Another crucial and well-known fact we must mention here is that $U_{n,0}$ acquires a natural $\mathbb{J}_1$-GH structure $\mathrm{D}$ from the action of the twisted de Rham operator $d_H=d-H\wedge$. In particular, $\mathrm{D}=\bar{\delta}_++\bar{\delta}_-$ according to the decomposition $L_1=L_+\oplus L_-$.

Now we assume that $M$ is spin and we then have two complex structures $J_\pm$ to construct a spinor bundle in the way described in \S~\ref{bi}:
\[S_1=\oplus_{p}(\wedge^pT^*_{+0,1}\otimes K_+^{1/2}),\quad S_2=\oplus_{p}(\wedge^pT^*_{-0,1}\otimes K_-^{1/2})\]
where $T^*_{\pm0,1}\subset T^*\otimes \mathbb{C}$ are the $-\sqrt{-1}$-eigenbundle of $J_\pm^*$ respectively. Note that the canonical orientations determined by $J_\pm$ need not to agree. These spinor bundles of course both acquire an Hermitian metric from the same Riemannian metric $g$ and a compatible spin connection by lifting $\nabla^-$. There is a $Cl(T, g)$-equivarriant isometry $p$ from $S_2$ to $S_1$, intertwining the spin connections. Such an isometry is unique up to an $S^1$-factor. Combining facts stated here and that in \S~\ref{bi}, we are led to an isomorphism
\[q^\flat:=\mathfrak{J}\circ \rho^{-1}\circ q^\sharp \circ(\textup{Id}\otimes p): S_1\otimes S_2\cong \mathcal{S}, \quad \phi\otimes \psi \mapsto q^\sharp(\phi\otimes p(\psi)).\]
\begin{proposition}The image of $K_+^{1/2}\otimes K_-^{1/2}\subset S_1\otimes S_2$ under the map $q^\flat$ is precisely $U_{n,0}$.
\end{proposition}
\begin{proof}Since $U_{n,0}$ is characterized by the property that it is annihilated by all elements in $L_1$, we only have to prove that $L_1$ also annilhilates $q^\flat(K_+^{1/2}\otimes K_-^{1/2})$. This is easy to check if one notes that $L_1=L_+\oplus L_-$ where a typical element in $L_\pm$ is of the form $X\pm g(X)$ for $X\in T_{\pm}^{0,1}$, and applies Prop.~\ref{bs1} to the present situation.
\end{proof}

The line bundle $K_+^{1/2}\otimes K_-^{1/2}$ has a biholomorphic structure: $K_+^{1/2}$ is $J_+$-holomorphic in a canonical way while the $(0,1)$-part of the Bismut connection $\nabla^-$ w.r.t. $J_+$ also gives rise to a $J_+$-holomorphic structure on $K_-^{1/2}$. Thus $K_+^{1/2}\otimes K_-^{1/2}$ acquires a $J_+$-holomorphic structure. In a similar way, it also acquires a $J_-$-holomorphic structure. The following fact was established in \cite{Gu3}.
\begin{proposition}The biholomorphic structure on $K_+^{1/2}\otimes K_-^{1/2}$ described above is actually a $\mathbb{J}_1$-GH structure, which is intertwined with the $\mathbb{J}_1$-GH structure on $U_{n,0}$ by the restriction of the map $q^\flat$.
\end{proposition}
\begin{proof}Note that $T\oplus T^*=V_+\oplus V_-$, where $V_\pm=(\textup{Id}\pm g)T$. We can $T$ identify with $V_\pm$ via projection and transport $\nabla^\pm$ to $V_\pm$ respectively. In this way, $T\oplus T^*$ is equipped with a connection which can be lifted to a spin connection $\nabla^{sp}$ on $\mathcal{S}$. If $\{e_i\}$ is a local orthonormal frame and $\{e^i\}$ its dual, then for $\Phi\in \Gamma(\mathcal{S})$,
\[\nabla_{e_i}^{sp} \Phi=\nabla^{LC}_{e_i}\Phi-\frac{1}{4}\sum_{j,k}H_{ijk}(e_je_k+e^je^k)\cdot \Phi.\]

Alternatively, this (generalized) spin connection can be interpreted from the bi-spinor side: We can equip $S_1$ with the spin connection lifted from $\nabla^+$ and $S_2$ with the spin connection lifted from $\nabla^-$. Then through the isomorphism $q^\flat$, this tensor product of connections gives rise to a connection on $\mathcal{S}$, which is precisely $\nabla^{sp}$. This can be easily checked by using Prop.~\ref{bs1}.

$\nabla^{sp}$ can be used to express the twisted de Rham operator $d_H$, just as in ordinary Riemannian geometry the de Rham operator can be conveniently expressed in terms of Levi-Civita connection. Actually, for $\Phi\in \Gamma(\mathcal{S})$,
\[d_H \Phi=\sum_ie^i\cdot\nabla^{sp}_{e_i}\Phi+\frac{1}{4}\sum_{i,j,k}H_{ijk}(\frac{1}{3}e^ie^je^k+e^ie_je_k)\cdot \Phi,\]
where $H_{ijk}=H(e_i, e_j, e_k)$. When restricted on $U_{n,0}$, the above formula will give an explicit expression of the GH structure in terms of $\nabla^{sp}$. More precisely, if $\{f^\pm_i\}$ are local unitary frames of $T_\pm^{1,0}$ respectively and $\Phi\in \Gamma(U_{n,0})$, then
\begin{equation}\bar{\delta}_+\Phi=\sum_i\frac{1}{2}(f_i^++g(f_i^+))\cdot\nabla^{sp}_{\bar{f}_i^+}\Phi-\frac{1}{4}\sum_{i,k}\overline{H^+_{ik\bar{k}}}(f_i^++g(f_i^+))\cdot \Phi \label{del}\end{equation}
and
\[\bar{\delta}_-\Phi=\frac{1}{2}\sum_i(-f_i^-+g(f_i^-))\cdot\nabla^{sp}_{\bar{f}_i^-}\Phi+\frac{1}{4}\sum_{i,k}\overline{H^-_{ik\bar{k}}}(-f_i^++g(f_i^-))\cdot \Phi\]
where $H^\pm_{ik\bar{k}}=H(f_i^\pm, f_k^\pm, \bar{f}_k^\pm)$.

Now let $\phi\in \Gamma(K_+^{1/2})$ and $\psi\in \Gamma(K_-^{1/2})$. Then $[\phi\otimes \psi]\in \Gamma(U_{n,0})$ and the right side of Eq.~(\ref{del}) is
\begin{equation}\frac{1}{2}\sum_i(f_i^++g(f_i^+))\cdot\{[(\nabla^+_{\bar{f}_i^+}-\frac{1}{2}\overline{H^+_{ik\bar{k}}})\phi\otimes \psi]+[\phi\otimes \nabla_{\bar{f}_i^+}^-\psi]\}.\label{messy}\end{equation}
If we can prove that the natural $J_+$-holomorphic structure on $K_+^{1/2}$ takes the following form
\begin{equation}\bar{\partial}\phi=\sum_ig(f_i^+)\wedge(\nabla^+_{\bar{f}_i^+}-\frac{1}{2}\overline{H^+_{ik\bar{k}}})\phi,\label{holo}\end{equation}
then Eq.~(\ref{del}) and Eq.~(\ref{messy}) precisely mean that $q^\flat$ intertwines the $J_+$-holomorphic structures on $K_+^{1/2}\otimes K_-^{1/2}$ and $U_{n,0}$. Eq.~(\ref{holo}) does hold: We can consider the Chern connection $\nabla^c$ on $T_+^{1,0}$ w.r.t. the natural $J_+$-holomorphic structure. $\nabla^c$ and $\nabla^+$ are related by (see \cite{Al} for example)
$$(\nabla^c_XY,Z)=(\nabla^+_XY,Z)-\frac{1}{2}H(X,J_+Y,J_+Z)-\frac{1}{2}H(X,Y,Z).$$
This identity is enough to verify Eq.~(\ref{holo}).

Similarly, $q^\flat$ intertwines the $J_-$-holomorphic structures on $K_+^{1/2}\otimes K_-^{1/2}$ and $U_{n,0}$. Since $\mathrm{D}=\bar{\delta}_++\bar{\delta}_-$, then \[(q^\flat)^{-1}\mathrm{D}q^\flat=(q^\flat)^{-1}\bar{\delta}_+q^\flat+(q^\flat)^{-1}\bar{\delta}_-q^\flat\] is a GH structure on $K_+^{1/2}\otimes K_-^{1/2}$.
\end{proof}
\emph{Remark}. Since $K_+\otimes K_-$ is the square of $K_+^{1/2}\otimes K_-^{1/2}$ and $K_+^{-1}\otimes K_-^{-1}$ is the dual of $K_+\otimes K_-$, $K_+\otimes K_-$ and $K_+^{-1}\otimes K_-^{-1}$ are thus GH line bundles in the natural way. This finally proves Prop.~\ref{holo2}.
\section{Scalar curvature for GK structures of symplectic type}
\subsection{GK structures of symplectic type and toric examples}\label{sym}
Recall that for the GK pair $(\mathbb{J}_1, \mathbb{J}_2)$ if $\mathbb{J}_2$ is a B-transform of a GC structure $\mathbb{J}_\omega$ induced from a symplectic form $\omega$, the GK manifold $(M, \mathbb{J}_1, \mathbb{J}_2)$ is said to be \emph{of symplectic type}. It is known from \cite{En} that for a given symplectic manifold $(M, \omega)$, compatible GC structures $\mathbb{J}_1$ which, together with a B-transform of $\mathbb{J}_\omega$, form GK structures on $M$ are in \emph{one-to-one} correspondence with \emph{tamed} integrable complex structures $J_+$ on $M$ whose \emph{symplectic adjoint} $J_+^{\omega}:=-\omega^{-1}J_+^*\omega$ is also integrable. This correspondence greatly facilitates the study of such structures. Precisely, if \emph{in the metric splitting} we set
\[\mathbb{J}_2=\frac{1}{2}\left(
  \begin{array}{cc}
    -J_++J_-& \omega_+^{-1}+\omega_-^{-1} \\
    -\omega_+-\omega_- & J_+^*-J_-^* \\
  \end{array}
\right)=\left(\begin{array}{cc}
1 & 0\\
-b & 1\\\end{array}\right)\left(\begin{array}{cc}
0 & \omega^{-1}\\
-\omega & 0\\\end{array}\right)\left(\begin{array}{cc}
1 & 0\\
b & 1\\\end{array}\right),\]
then the following basic identities can be easily obtained:
\begin{equation}J_-=J_+^\omega=-\omega^{-1} J_+^*\omega,\quad g=-\frac{1}{2}\omega (J_++J_-),\quad b=-\frac{1}{2}\omega (J_+-J_-).\label{sy}\end{equation}
In particular, the curvature of the metric splitting is $H=db$.

For a GK manifold of symplectic type, another natural splitting is also often used in the literature. In this splitting, the $\sqrt{-1}$-eigenbundle of $\mathbb{J}_2$ is of the form $(\textup{Id}+\sqrt{-1}\omega)(T\otimes \mathbb{C})$, or the generalized pure spinor of $\mathbb{J}_2$ is $e^{-\sqrt{-1}\omega}$. We call this splitting the \emph{symplectic splitting}. Obviously, it relates to the metric splitting by the 2-form $b$ and in the latter the generalized pure spinor is changed into $e^{b-\sqrt{-1}\omega}$.

There are many examples of GK manifolds of symplectic type coming from toric geometry.

\begin{definition}A toric symplectic manifold $(M, \omega, \mathbb{T}, \mu)$ of dimension $2n$ is a compact connected symplectic manifold $(M, \omega)$ with an effective and Hamiltonian action of the $n$-dimensional torus $\mathbb{T}=\mathbb{T}^n$ (with Lie algebra $\mathfrak{t}$). Note that here $\mu: M\rightarrow \mathfrak{t}^*$ is the moment map.
\end{definition}
Let $P$ be the image of $\mu$. Then by the famous convexity theorem proved by Atiyah \cite{At} and Guillemin-Sterberg \cite{GS}, $P$ is a polytope which is commonly called the moment polytope of the Hamiltonian action. T. Delzant proved that compact toric symplectic manifolds are actually classified by their moment polytopes \cite{Del}. A byproduct of Delzant's theorem is that any toric symplectic manifold admits a compatible invariant K$\ddot{a}$hler structure. For a fixed toric symplectic manifold, Guillemin found in \cite{Gul} that a compatible toric K$\ddot{a}$hler structure can be efficiently described by a strictly convex smooth function $\tau$ defined in the interior $\mathring{P}$ of $P$. This function is called the \emph{symplectic potential} in the literature. Through the works \cite{Bou, Wang1, Wang2}, the main body of Guillemin's theory has been extended to the setting of toric GK structures of symplectic type---these are GK structures whose underlying $J_+$ is invariant under the torus action.

Let $\{t_i\}$ be a fixed basis of $\mathfrak{t}$ and $X_i$ the corresponding fundamental vector fields over $M$. A general toric GK structure of symplectic type on $(M, \omega, \mathbb{T}, \mu)$ can be viewed as constructed in the following steps.

\emph{Firstly} a compatible toric K$\ddot{a}$hler structure on $M$ is chosen. Let $I$ be the underlying invariant complex structure and $\tau$ its associated symplectic potential---a strictly convex function on $\mathring{P}$ satisfying certain asymptotic conditions when approaching the boundary of $P$. The vector fields $\{X_i, IX_i\}$ form an invariant frame of $T\mathring{M}$ where $\mathring{M}=\mu^{-1}(\mathring{P})\cong \mathbb{T}\times \mathring{P}$. Let $\{\zeta_i, u_i\}$ be the corresponding dual frame of $T^*\mathring{M}$. Then in terms of $\zeta_i, d\mu_i$, $I$ takes the following form (we have written $\zeta_i$ and $d\mu_i$ in a column)
\[I^*\left(
       \begin{array}{c}
         \zeta \\
         d\mu \\
       \end{array}
     \right)=\left(
               \begin{array}{cc}
                 0 & \phi_s \\
                 -(\phi_s)^{-1} & 0 \\
               \end{array}
             \right)\left(
       \begin{array}{c}
         \zeta \\
         d\mu \\
       \end{array}
     \right),
\]
where $\phi_s$ is the Hessian of $\tau$. In particular, the K$\ddot{a}$hler metric is of the form \[g_0=(\phi_s)^{-1}_{ij}\zeta_i\otimes \zeta_j+\phi_{sij}d\mu_i\otimes d\mu_j.\] Note that $\{\zeta_i\}$ is actually a flat connection on the trivial torus bundle $\mu: \mathring{M}\rightarrow \mathring{P}$ and thus locally $\zeta_i=d\theta_i$ for local functions $\theta_i$. Actually, $\{\theta_i, \mu_i\}$ provides a Darboux coordinate chart for $\omega$, i.e. locally $\omega= d\mu_i\wedge d\theta_i$. Such a coordinate chart is also called \emph{admissible} by Boulanger in \cite{Bou}.

\emph{Secondly}, choose two constant anti-symmetric $n\times n$ real matrices $C, F$ such that the matrix-valued function
\[\textup{I}+\frac{1}{4}[(\phi_s)^{-1/2}F(\phi_s)^{-1/2}]^2\]
is positive-definite on $P$, where $\textup{I}$ is the identity matrix. One should note that no further requirement for $C$ is needed here.

\emph{Thirdly}, the matrix $F$ is used to construct two other flat connections $\zeta^\pm$:
\[\zeta^\pm=\zeta\mp\frac{1}{2}Fd\mu.\]
Then two invariant complex structures $J_\pm$ can be specified on $\mathring{M}$ if we insist
\[J_+^*\left(
       \begin{array}{c}
         \zeta^+ \\
         d\mu \\
       \end{array}
     \right)=\left(
               \begin{array}{cc}
                 0 & \phi^T \\
                 -(\phi^T)^{-1} & 0 \\
               \end{array}
             \right)\left(
       \begin{array}{c}
         \zeta^+ \\
         d\mu \\
       \end{array}
     \right)\]
     and
     \[J_-^*\left(
       \begin{array}{c}
         \zeta^- \\
         d\mu \\
       \end{array}
     \right)=\left(
               \begin{array}{cc}
                 0 & \phi \\
                 -\phi^{-1} & 0 \\
               \end{array}
             \right)\left(
       \begin{array}{c}
         \zeta^- \\
         d\mu \\
       \end{array}
     \right)\]
     where $\phi=\phi_s+C$ and $\phi^T$ is the transpose of $\phi$. Both $J_+$ and $J_-$ thus defined can be smoothly extended to the whole of $M$. They together with the symmetric part of $-\omega J_+$ form the biHermitian triple defining a GK structure.

     Actually it was proved in \cite{Wang1, Wang2} that any toric GK structure of symplectic type arises in the manner described as above. Note that if $F=0$, then the underlying GK structure is called \emph{anti-diagonal}.

     For later convenience, we collect from \cite{Wang2} the matrix forms of several structures viewed as linear maps in the frame  $\{\partial_{\theta_i}, \partial_{\mu_i}\}$:
     \[J_+\sim \left(
            \begin{array}{cc}
              \phi^{-1}F/2 & -\phi^{-1} \\
              \phi+1/4F\phi^{-1}F & -F\phi^{-1}/2 \\
            \end{array}
          \right),\quad J_-\sim \left(
            \begin{array}{cc}
              -(\phi^T)^{-1}F/2& -(\phi^T)^{-1} \\
              \phi^T+1/4F(\phi^T)^{-1}F & F(\phi^T)^{-1}/2 \\
            \end{array}
          \right),\]
\[g\sim\left(
         \begin{array}{cc}
           (\phi^{-1})_s & (\phi^{-1})_aF/2 \\
            F(\phi^{-1})_a/2& \phi_s+\frac{1}{4}F(\phi^{-1})_sF \\
         \end{array}
       \right),\quad b\sim \left(
                             \begin{array}{cc}
                               (\phi^{-1})_a & (\phi^{-1})_sF/2 \\
                               F(\phi^{-1})_s/2 & \phi_a+\frac{1}{4}F(\phi^{-1})_aF \\
                             \end{array}
                           \right).
\]
          \[(\frac{J_++J_-}{2})^{-1}\sim \left(
       \begin{array}{cc}
         \frac{1}{2}\Xi^{-1}F\phi_0 & \Xi^{-1} \\
         -\phi^T(\phi_s)^{-1}\phi+\frac{1}{4}\phi_0^TF\Xi^{-1}F\phi_0 & \frac{1}{2}\phi_0^TF\Xi^{-1} \\
       \end{array}
     \right),\]
     where $\phi_0=(\phi_s)^{-1}\phi_a$ and we denote $C$ by $\phi_a$ to emphasize that it is the anti-symmetric part of $\phi$.
\subsection{Goto's approach to scalar curvature and its refinement}\label{goto}
In \cite{Go1}, Goto provided a notion of scalar curvature for GK manifolds of symplectic type in terms of generalized pure spinors defining the underlying geometry. The goal of this section is thus two-fold: on one side we briefly recall Goto's definition and on the other side we refine it in such a way that, with the help of bi-spinors, we can compute the scalar curvature in terms of the underlying biHermitian data.

  In \cite{Go1}, Goto actually used a different splitting which was neither the metric splitting nor the symplectic splitting. The generalized pure spinor of $\mathbb{J}_2$ was chosen to be $\Psi=e^{B-\sqrt{-1}\omega}$ where $B$ is a closed real 2-form. In this setting, let $\Phi_\alpha$ be a local frame of $U_{n,0}$. Then it is well-known that
\[d\Phi_\alpha=\eta_\alpha\cdot \Phi_\alpha\]
for some real generalized vector field $\eta_\alpha$, which is uniquely determined by the $\mathbb{J}_1$-GH structure on $U_{n,0}$ and the choice of $\Phi_\alpha$. Let $\rho_\alpha$ be the real function\footnote{$\rho_\alpha$ is actually positive because $\Phi_\alpha$ and $\Psi$ give the same orientation on $M$.} defined by
\begin{equation}(\Phi_\alpha, \bar{\Phi}_\alpha)_{Ch}=\rho_\alpha(\Psi, \bar{\Psi})_{Ch}.\label{metric}\end{equation}
Then we have the differential form
\[d[(-\mathbb{J}_1\eta_\alpha+\frac{1}{2}\mathbb{J}_1d\log \rho_\alpha)\cdot \Psi]=-(P_1+\sqrt{-1}P_2)\Psi\]
where $P_1$, $P_2$ are real closed 2-forms. This expression is actually globally well-defined. Goto called $P_1$ the (generalized) Ricci form and
\begin{equation}\kappa:=\frac{2nP_1\wedge \omega^{n-1}}{\omega^n}\label{sc}\end{equation}
the (generalized) scalar curvature. Later in \cite{Go2}, the above seemingly strange expression of Ricci curvature was interpreted in terms of generalized connection. Actually, the above data $\rho_\alpha$ gives rise to an Hermitian metric $h_\omega$ on $U_{n,0}$. This metric and the $\mathbb{J}_1$-GH structure determine a canonical generalized connection $D^\omega$. The generalized vector field
\[\sigma_\alpha=-\sqrt{-1}(-\mathbb{J}_1\eta_\alpha+\frac{1}{2}\mathbb{J}_1d\log \rho_\alpha)\]
is precisely the connection form of $D^\omega$ in the normalized local frame $\Phi_\alpha/\sqrt{\rho_\alpha}$.

We shall give here a refinement of Goto's definition of Ricci curvature in terms of ordinary connections.

\emph{Firstly}, we would like to drop the real 2-form $B$ and thus use the symplectic splitting. From the viewpoint of biHermitian geometry, this freedom parameterized by $B$ only arises when one reformulates the geometry using the language of GC geometry and thus can be viewed as a gauge freedom. By setting $B=0$, we see that $\sqrt{-1}P_1$ is nothing else but the differential of the 1-form part $\sigma_\alpha'$ of the generalized vector field $\sigma_\alpha$ because the vector filed part of $\sigma_\alpha$ only contributes to $P_2$. However, $\sigma_\alpha'$ is just the connection 1-form of the ordinary connection part $\nabla^\omega$ of $D^\omega$ in the symplectic splitting and in the local frame $\Phi_\alpha/\sqrt{\rho_\alpha}$.

Note that $U_{n,0}$ inherits another Hermitian metric from the Hodge metric on $\mathcal{S}$. To distinguish this metric with that defined by Eq.~(\ref{metric}), we call the latter the \emph{symplectic metric}.
To summarize our argument up to now, we have
\begin{definition}Let $D^\omega$ be the canonical generalized connection on $U_{n,0}$ determined by the natural $\mathbb{J}_1$-GH structure and the symplectic metric $h_\omega$, and let $\nabla^\omega$ be the ordinary connection part of $D^\omega$ in the symplectic splitting. Denote the curvature of $\nabla^\omega$ by $R^\omega$. We call $P_1=-\sqrt{-1}R^\omega$ the (symplectic) Ricci form and call the function $\kappa$ determined by Eq.~(\ref{sc}) the (symplectic) scalar curvature.
\end{definition}
\emph{Remark}. Of course, the Ricci form and the scalar curvature can be defined in any splitting in a similar way and thus they both are generally splitting-relevant.

\emph{Secondly}, from the discussion of \S~\ref{spGK}, we can identify $K_+^{1/2}\otimes K_-^{1/2}$ with $U_{n,0}$ via the map $q^\flat$. Even more, this map intertwines GH structures and metrics induced from $g$ on these two line bundles. If we can further know how the Hodge metric and the symplectic metric on $U_{n,0}$ are related, it is then possible to express the scalar curvature in terms of the biHermitian data directly. Note that from arguments in \S~\ref{spGK}\[\star\cdot \bar{\Phi}_\alpha=(-\sqrt{-1})^{-n}\bar{\Phi}_\alpha,\quad (\Psi, \bar{\Psi})_{Ch}=(-2\sqrt{-1})^nvol_\omega,\]
where $vol_\omega$ is the Liouville volume element $\omega^n/n!$.
Then Eq.~(\ref{metric}) is precisely
\[(-\sqrt{-1})^n(\Phi_\alpha, \Phi_\alpha)_hvol_g=\rho_\alpha\times (-2\sqrt{-1})^nvol_\omega,\]
or simply for a constant $c$
\begin{equation}
\rho_\alpha=c\times (\Phi_\alpha, \Phi_\alpha)_h\times \frac{vol_g}{vol_\omega}.
\end{equation}
This is precisely how the two metrics are related. Notice that the constant $c$ is irrelevant for computing the Ricci form.

Now we are in a position to derive a local formula for the Ricci form in terms of the biHermitian data. Let $\{z_i^\pm\}$ be local $J_\pm$-holomorphic coordinates in the same chart. Then locally
\[g=h^+_{i\bar{j}}dz^+_i\otimes d\bar{z}_j^+=h^-_{i\bar{j}}dz^-_i\otimes d\bar{z}_j^-.\]
Let
\[\gamma:=\frac{vol_g}{\sqrt{\det h^+\times \det h^-}\times vol_\omega}\]
and $\sigma_\pm$ be the connection forms of the $J_\pm$-holomorphic structures in the local frame of bi-spinors
\[s=(dz_1^+\wedge\cdots\wedge dz_n^+)^{1/2}\otimes (dz_1^-\wedge\cdots\wedge dz_n^-)^{1/2},\]
respectively, i.e. $\bar{\partial}_\pm s=\sigma_\pm s$. Note that $\gamma$ is $h_\omega(s,s)$ up to a constant factor.
\begin{proposition}In terms of the above local data, the Ricci form is
\[-\frac{\sqrt{-1}}{2}d\{(\partial_++\partial_-)\log \gamma+\sigma_++\sigma_--\bar{\sigma}_+-\bar{\sigma}_-+bg^{-1}[(\partial_+-\partial_-)\log \gamma+\sigma_+-\bar{\sigma}_+-\sigma_-+\bar{\sigma}_-]\}\]
where $\partial_\pm\log \gamma$ are the (1,0)-part of $d\log \gamma$ w.r.t. $J_\pm$ respectively.
\end{proposition}
\begin{proof}Note that \emph{in the metric splitting},
\[D^\omega=\frac{1}{2}(\nabla_++\nabla_-)+\frac{1}{2}g^{-1}(\nabla_+-\nabla_-)\]
where $\nabla_\pm$ are the Chern connections associated to the $J_\pm$-holomorphic structures respectively. Thus \emph{in the symplectic splitting}, the ordinary connection part of $D^\omega$ is
\[\nabla^\omega=\frac{1}{2}(\nabla_++\nabla_-)+\frac{1}{2}bg^{-1}(\nabla_+-\nabla_-).\]
Since in the local frame $s$ these Chern connections are classically determined by
\[\nabla_\pm s=(\partial_\pm \log \gamma+\sigma_\pm-\bar{\sigma}_\pm)s,\]
our conclusion then follows.
\end{proof}
\emph{Remark}. In some cases, the above seemingly messy formula can be simplified greatly: If we rescale $s$ with a smooth function $\varpi$ such that $\varpi s$ is biholomorphic (this is always possible at least around a regular point of $\mathbb{J}_1$) and consequently $\sigma_\pm\equiv 0$, then the Ricci form is
\[-\frac{\sqrt{-1}}{2}d[(\partial_++\partial_-)\log (|\varpi|^2\gamma)+bg^{-1}(\partial_+-\partial_-)\log (|\varpi|^2\gamma)].\]
This will be the form we shall apply to toric GK structures of symplectic type in the next section.

\section{scalar curvature for toric GK structures of symplectic type}\label{sec5}
\subsection{Boulanger's approach to scalar curvature}\label{bou}
Following the principle of scalar curvature as a moment map, L. Boulanger in \cite{Bou} defined the scalar curvature of a toric GK manifold of symplectic type formally as the moment map of a Hamiltonian action of an infinite-dimensional Lie group. In this subsection, for the reader's convenience, we sketch this formalism briefly. While Boulanger only gave an explicit expression of the scalar curvature in the anti-diagonal case in terms of the symplectic potential, we shall give such an expression for a \emph{generic} toric GK manifold of symplectic type.

Let $(M, \omega, \mathbb{T}, \mu)$ be a toric symplectic manifold and denote the space of all invariant \emph{almost} GK structures of symplectic type (with symplectic form $\omega$) by $\mathcal{M}$, which is formally an infinite-dimensional manifold. Elements in $\mathcal{M}$ are not necessarily integrable. However, such structures $(\mathbb{J}_1, \mathbb{J}_2)$ are in one-to-one correspondence with invariant \emph{almost} complex structures $J_+$ tamed with $\omega$. Define
\[A:=-2(J_++J_-)^{-1},\quad B:=-(J_+-J_-)(J_++J_-)^{-1},\quad K:=A+\sqrt{-1}B,\]
where $J_-$ is the symplectic adjoint of $J_+$, i.e. $J_-=-\omega^{-1}J_+^*\omega$.
Note that $K$ is a homomorphism of $T\otimes \mathbb{C}$ satisfying the following algebraic conditions:
\[K^2=-\textup{Id},\quad K^\omega=K.\]
Such $K$'s can be used to parameterize elements in $\mathcal{M}$. Consequently, at $K\in \mathcal{M}$, the tangent space is
\[T_K\mathcal{M}=\{\dot{K}\in \textup{Hom}(T\otimes \mathbb{C})|K\dot{K}+\dot{K}K=0, \dot{K}^\omega=\dot{K}\}.\]
Note that the condition of tameness has no constraint on a tangent vector in $T_K\mathcal{M}$ because this is an open condition.

$\mathcal{M}$ is formally an infinite-dimensional K$\ddot{a}$hler manifold. The complex structure $\mathbb{K}$ on $\mathcal{M}$ is defined by
\[\mathbb{K}\dot{K}=K\dot{K},\quad \forall \dot{K}\in T_K\mathcal{M},\]
and the symplectic form on $\mathcal{M}$ is defined by
\[\Omega_K(\dot{K}_1, \dot{K}_2)=\frac{1}{2}\int_M\textup{Tr}(K\dot{K}_1\dot{K}_2)vol_\omega,\quad \forall \dot{K}_1, \dot{K}_2\in T_K\mathcal{M}. \]

Denote $\textup{Ham}_c^\mathbb{T}(M, \omega)$ the group of Hamiltonian diffeomorphisms generated by invariant functions with zero mean supported in $\mathring{M}$. Elements in $\textup{Ham}_c^\mathbb{T}(M, \omega)$ act on $\mathcal{M}$ by conjugation: $\Upsilon \cdot K=\Upsilon_*K\Upsilon^{-1}_*$. The set $C_c^\mathbb{T}(M)$ of invariant smooth functions with zero mean supported in $\mathring{M}$ can be viewed as the Lie algebra of $\textup{Ham}_c^\mathbb{T}(M, \omega)$, with the Poisson bracket of functions as its Lie bracket. Then the fundamental vector filed $\mathbb{V}_f$ on $\mathcal{M}$ generated by $f\in C_c^\mathbb{T}(M)$ is
\[\mathbb{V}_f(K)=-L_{V_f}K,\quad K\in \mathcal{M},\]
where $L_{V_f}$ is the Lie derivative along the Hamiltonian vector field $V_f=-\omega^{-1}df$. More technical details concerning these formally defined structures can be found in Boulanger's Ph.D thesis \cite{Bou2}.

Let us choose admissible coordinates $\{\theta_i, \mu_i\}$. Using integration by part Boulanger proved the following
\begin{theorem} The action of $\textup{Ham}_c^\mathbb{T}(M, \omega)$ on $\mathcal{M}$ is Hamiltonian with the following moment map
\[\nu^f(K)=-\int_Mf\kappa vol_\omega, \]
where $$\kappa=\sum_{i,j}\frac{\partial^2Z_{ij}}{\partial\mu_i\partial \mu_j}$$
with $Z_{ij}=\omega(\partial_{\theta_i}, A\partial_{\theta_j})$.
\end{theorem}
 Boulanger then called $\kappa$ the scalar curvature of the almost GK structure parameterized by $K$. In particular, if $\{\theta_i, \mu_i\}$ are the admissible coordinates underlying the GK structure of symplectic type described in \S~\ref{sym}, we have
 \begin{corollary}The scalar curvature of a toric GK structure of symplectic type described in \S~\ref{sym} is
 \begin{equation}\kappa=-\sum_{ij}\frac{\partial^2(\Xi^{-1})_{ij}}{\partial\mu_i\partial \mu_j},\label{sca1}\end{equation}
 where $\Xi=\phi_s+\frac{1}{4}F(\phi_s)^{-1}F$ and $\phi_s$ is the Hessian of the symplectic potential $\tau$.
  \end{corollary}
  \begin{proof}
  We only need to note that in the admissible coordinates $\{\theta_i, \mu_i\}$,
  \[A\partial_\theta=-\frac{1}{2}\Xi^{-1}F\phi_0\partial_\theta-\Xi^{-1}\partial_\mu\]
  where $\phi_0=(\phi_s)^{-1}\phi_a$.
  \end{proof}
  \emph{Remark}. If $F=0$, in other words in the anti-diagonal case, then $\Xi=\phi_s$ and we recover Boulanger's expression of scalar curvature for this case. This is actually the scalar curvature of the underlying toric K$\ddot{a}$hler structure. It is remarkable that between the two constant matrices $C$ and $F$ only the latter has contribution to the scalar curvature.

\subsection{Scalar curvature computed via Goto's formalism}\label{goto2}
Now due to the analysis in \S~\ref{goto}, it is conceptually quite clear how to compute the scalar curvature for a toric GK manifold $M$ of symplectic type. However, we will use $K_+^{-1}\otimes K_-^{-1}$ rather than $K_+^{1/2}\otimes K_-^{1/2}$ as our basic object to be analyzed. $K_+^{-1}\otimes K_-^{-1}$ is the square of the  dual of the canonical line bundle $K_+^{1/2}\otimes K_-^{1/2}$. The GH structure and symplectic metric on the latter can be naturally inherited by the former. We can compute the canonical generalized connection on $K_+^{-1}\otimes K_-^{-1}$ to get the Ricci form. This is much more in the original spirit of K$\ddot{a}$hler geometry---after all in K$\ddot{a}$hler geometry Ricci form is the first Chern form of the anti-canonical line bundle. We will work on the open and dense set $\mathring{M}$ and try to find a GH section of $K_+^{-1}\otimes K_-^{-1}$, i.e. a biholomorphic section. Notation in \S~\ref{sym} will continue to be used throughout this subsection.

Let us introduce two else coordinate systems first. Since $\zeta^\pm_i$ are flat connections, locally $\zeta^\pm_i=d\theta_i^\pm$ for some functions $\theta_i^\pm$. Then $\{\theta_i^\pm, \mu_i\}$ can be used as coordinates on $\mathring{M}$. In particular, $\zeta_i^\pm-\sqrt{-1}J_\pm^*\zeta_i^\pm$ are of the form $dz_i^\pm$ where $\{z_i^\pm\}$ can be viewed as $J_\pm$-holomorphic coordinates on $\mathring{M}$ respectively.

In terms of the coordinates $\{\theta_i^+, \mu_i\}$, we have\footnote{Note that in the coordinate systems $\{\theta_i^\pm, \mu_i\}$, $\partial_{\mu_i}$ actually represents different vector fields. Thus we use $\partial_{\mu_i^\pm}$ to distinguish them. However, we have $\partial_{\theta_i^\pm}=\partial_{\theta_i}=X_i$.}
\[\partial_{z_i^+}=\frac{1}{2}(\partial_{\theta_i^+}+\sqrt{-1}(\phi^{-1})_{ji}\partial_{\mu_j^+}),\]
and similarly, in terms of $\{\theta_i^-, \mu_i\}$,
\[\partial_{z_i^-}=\frac{1}{2}(\partial_{\theta_i^-}+\sqrt{-1}(\phi^{-1})_{ij}\partial_{\mu_j^-}).\]
With $\{\partial_{\theta_i^+}, \partial_{\mu_i^+}\}$ the Riemannian metric has the following matrix form \cite{Wang2}
\[g\sim \left(
           \begin{array}{cc}
             (\phi^{-1})_s & \phi^{-1}F/2 \\
             -F(\phi^{T})^{-1}/2 & \phi_s\\
           \end{array}
         \right).\]
Then it can be computed directly that\footnote{We have viewed $g$ as a linear map and the notation here is thus slightly different from the common convention.}
\[\det g=\det(\phi^{-1})_s\times \det \Xi,\]
\[g_{\bar{j}i}=\frac{1}{2}[\phi^{-1}(\phi_s-\frac{\sqrt{-1}}{2}F)(\phi^T)^{-1}]_{ji},\]
\[\det(g_{\bar{j}i})=(1/2)^n\det(\phi_s-\frac{\sqrt{-1}}{2}F)\times[\det(\phi^{-1})]^2.\]
\begin{lemma}\label{det}\[\det(\phi_s-\frac{\sqrt{-1}}{2}F)=\det(\phi_s+\frac{\sqrt{-1}}{2}F)=(\det \phi_s \times \det \Xi)^{1/2}.\]
\end{lemma}
\begin{proof}Consider the real matrix
\[M_1=\left(
    \begin{array}{cc}
      \phi_s & \frac{F}{2} \\
      -\frac{F}{2} & \phi_s \\
    \end{array}
  \right),
\]
which commutes with the matrix
\[J=\left(
    \begin{array}{cc}
      0 & \textup{I} \\
      -\textup{I} & 0 \\
    \end{array}
  \right).
\]
Then $M_2:=\phi_s+\frac{\sqrt{-1}}{2}F$ is nothing else but the restriction of $M_1$ on the $\sqrt{-1}$-eigenspace $W\cong \mathbb{C}^n\subset \mathbb{R}^{2n}\otimes \mathbb{C}$ of $J$ in terms of the standard basis of $\mathbb{C}^n$. It is elementary to find $\det M_1=(\det(\phi_s+\frac{\sqrt{-1}}{2}F))^2$ and thus our formula can be derived.
\end{proof}

We are especially interested in how $\nabla^+$ acts on $\partial_{z_j^+}$.
\begin{lemma}\label{Bismut2}Let $\{\theta_i, \mu_i\}$ be the admissible coordinates associated to the underlying toric K$\ddot{a}$hler structure. Then \[\nabla^+_{\partial_{\theta_i}}\partial_{z_j^+}=\frac{\sqrt{-1}}{4}(\phi^{-1})_{ij,k}(\phi^{-1})_{kl}g^{-1}d\bar{z}_l^+\]
and
\begin{eqnarray*}\nabla^+_{\partial_{\mu_i}}\partial_{z_j^+}&=&\frac{1}{4}(\phi^{-1})_{kj,i}[\delta_{kl}+\frac{\sqrt{-1}}{2}(\phi^{-1}F)_{kl}]g^{-1}d\bar{z}_l^+\\
&-&\frac{\sqrt{-1}}{8}[(\phi^{-1}F)_{kj,i}-(\phi^{-1}F)_{ij,k}](\phi^{-1})_{kl}g^{-1}d\bar{z}_l^+,\end{eqnarray*}
where $\delta_{ij}$ is the Kronecker delta.
\end{lemma}
\begin{proof}Note that in the admissible coordinates, the matrix form of $g$ and $b$ take the following forms
\[g\sim\left(
         \begin{array}{cc}
           (\phi^{-1})_s & (\phi^{-1})_aF/2 \\
            F(\phi^{-1})_a/2& \phi_s+\frac{1}{4}F(\phi^{-1})_sF \\
         \end{array}
       \right),\quad b\sim \left(
                             \begin{array}{cc}
                               (\phi^{-1})_a & (\phi^{-1})_sF/2 \\
                               F(\phi^{-1})_s/2 & \phi_a+\frac{1}{4}F(\phi^{-1})_aF \\
                             \end{array}
                           \right).
\]
Explicitly, we have
\[b=\frac{1}{2}(\phi^{-1})_{aij}d\theta_jd\theta_i+\frac{1}{2}[(\phi^{-1})_sF]_{ij}d\theta_jd\mu_i+\frac{1}{2}[\phi_a+\frac{1}{4}F(\phi^{-1})_aF]_{ij}d\mu_jd\mu_i.\]

Now the curvature 3-form in the metric splitting is $H=db$. We want to find explicit expressions for $H(\partial_{\theta_i}, \partial_{\theta_j})$ and $H(\partial_{\mu_i}, \partial_{\theta_j})$, and it is clear that only the first two terms of the above expression of $b$ contribute:
\[H(\partial_{\theta_i}, \partial_{\theta_j})=-(\phi^{-1})_{aij,k}d\mu_k,\]
and
\[H(\partial_{\mu_i}, \partial_{\theta_j})=-(\phi^{-1})_{ajk,i}d\theta_k+\frac{1}{2}\{[(\phi^{-1})_sF]_{kj,i}-[(\phi^{-1})_sF]_{ij,k}\}d\mu_k.\]
It can be computed easily, using the general expression of Christoffel coefficients
\[\Gamma_{ij}^k=\frac{1}{2}G^{kl}(G_{il,j}+G_{jl,i}-G_{ij,l}),\]
that
\[\nabla^{LC}_{\partial_{\theta_i}}\partial_{\theta_j}=-\frac{1}{2}(\phi^{-1})_{sij,k}g^{-1}d\mu_k,\]
and that
\[\nabla^{LC}_{\partial_{\mu_i}}\partial_{\theta_j}=\frac{1}{2}(\phi^{-1})_{sjk,i}g^{-1}d\theta_k+\frac{1}{4}\{[(\phi^{-1})_aF]_{kj,i}-[(\phi^{-1})_aF]_{ij,k}\}g^{-1}d\mu_k.\]
Combining the above together, we obtain
\[\nabla^+_{\partial_{\theta_i}}\partial_{\theta_j}=-\frac{1}{2}(\phi^{-1})_{ij,k}g^{-1}d\mu_k\]
and
\[\nabla^+_{\partial_{\mu_i}}\partial_{\theta_j}=\frac{1}{2}(\phi^{-1})_{kj,i}g^{-1}d\theta_k+\frac{1}{4}[(\phi^{-1}F)_{kj,i}-(\phi^{-1}F)_{ij,k}]g^{-1}d\mu_k.\]

Now let us turn to the computation of $\nabla^+\partial_{z_j^+}$. Note that
\[(\textup{Id}+\sqrt{-1}J_+^*)d\mu=-\sqrt{-1}(\phi^{-1})^Td\bar{z}^+\]
and
\begin{eqnarray*}(\textup{Id}+\sqrt{-1}J_+^*)d\theta&=&(\textup{Id}+\sqrt{-1}J_+^*)(d\theta^++\frac{1}{2}Fd\mu)\\
&=&(\textup{I}-\frac{\sqrt{-1}}{2}F(\phi^{-1})^T)d\bar{z}^+.\end{eqnarray*}
Since \[\partial_{z_i^+}=\frac{1}{2}(\textup{Id}-\sqrt{-1}J_+)\partial_{\theta_i}\] and that $J_+$ is $\nabla^+$-flat and $g$-compatible, we have
\[\nabla^+_{\partial_{\theta_i}}\partial_{z_j^+}=\frac{\sqrt{-1}}{4}(\phi^{-1})_{ij,k}(\phi^{-1})_{kl}g^{-1}d\bar{z}_l^+\]
and
\begin{eqnarray*}\nabla^+_{\partial_{\mu_i}}\partial_{z_j^+}&=&\frac{1}{4}(\phi^{-1})_{kj,i}[\delta_{kl}+\frac{\sqrt{-1}}{2}(\phi^{-1}F)_{kl}]g^{-1}d\bar{z}_l^+\\
&-&\frac{\sqrt{-1}}{8}[(\phi^{-1}F)_{kj,i}-(\phi^{-1}F)_{ij,k}](\phi^{-1})_{kl}g^{-1}d\bar{z}_l^+.\end{eqnarray*}
\end{proof}
\begin{lemma}\label{holos}Let $\epsilon=[\frac{\det(\phi_s\Xi)}{(\det \phi)^2}]^{-1/2}$ and $s_0^+=\partial_{z_1^+}\wedge\cdots \wedge\partial_{z_n^+}$. Then $\epsilon s_0^+$ is a $J_-$-holomorphic section of $K_+^{-1}$.
\end{lemma}
\begin{proof}
Let $g^{-1}d\bar{z}_l^+=Q_{ml}\partial_{z_m^+}$. Then
\begin{equation}Q=2\phi^T(\phi_s-\frac{\sqrt{-1}}{2}F)^{-1}\phi.\label{Q}\end{equation}
Therefore, by Lemma~\ref{Bismut2} the connection 1-form $\sigma$ of $\nabla^+$ in the frame $s_0^+$ of $K_+^{-1}$ is
\begin{eqnarray*}\sigma=\frac{Q_{jl}}{4}\times\{\sqrt{-1}(\phi^{-1})_{ij,k}(\phi^{-1})_{kl}d\theta_i+(\phi^{-1})_{kj,i}[\delta_{kl}+\frac{\sqrt{-1}}{2}(\phi^{-1}F)_{kl}]d\mu_i\\
-\frac{\sqrt{-1}}{2}[(\phi^{-1}F)_{kj,i}-(\phi^{-1}F)_{ij,k}](\phi^{-1})_{kl}d\mu_i\}.\end{eqnarray*}

We want to express $\sigma(\partial_{\bar{z}_n^-})$ in the form of $-\partial_{\bar{z}_n^-}\log \epsilon$ for a real function $\epsilon$ to be determined. If this holds, then
\[\nabla^+_{\partial_{\bar{z}_n^-}}(\epsilon s_0^+)=\partial_{\bar{z}_n^-}\epsilon \times s_0^+-\epsilon \partial_{\bar{z}_n^-}\log \epsilon \times s_0^+=0, \]
i.e., $\epsilon s_0^+$ is a $J_-$-holomorphic section. This is possible because the $(0,1)$-part of $\nabla^+$ w.r.t. $J_-$ gives rise to a $J_-$-holomorphic structure on $K_+^{-1}$.

Note that
\begin{eqnarray*}\partial_{\bar{z}^-}&=&\frac{1}{2}(\partial_{\theta^-}+\sqrt{-1}J_-\partial_{\theta^-})
=\frac{1}{2}(\partial_{\theta}+\sqrt{-1}J_-\partial_{\theta})\\&=&\frac{1}{2}(\phi^{-1})^T[(\phi^T-\frac{\sqrt{-1}}{2}F)\partial_\theta-\sqrt{-1}\partial_\mu]
\end{eqnarray*}
and thus \[\partial_{\bar{z}_n^-}\log \epsilon=-\frac{\sqrt{-1}}{2}(\phi^{-1})_{np}(\log \epsilon)_{,p}.\]
Now
\begin{eqnarray*}
-8\sqrt{-1}\sigma(\partial_{\bar{z}_n^-})&=&\{(\phi+\frac{\sqrt{-1}}{2}F)_{pq}(\phi^{-1})_{qj,k}(\phi^{-1})_{kl}-(\phi^{-1})_{kj,p}[\delta_{kl}+\frac{\sqrt{-1}}{2}(\phi^{-1}F)_{kl}] \\
&+&\frac{\sqrt{-1}}{2}[(\phi^{-1}F)_{kj,p}-(\phi^{-1}F)_{pj,k}](\phi^{-1})_{kl} \}\times Q_{jl}\times (\phi^{-1})_{np}.
\end{eqnarray*}
Note that
\[(\phi+\frac{\sqrt{-1}}{2}F)_{pq}(\phi^{-1})_{qj,k}=\frac{\sqrt{-1}}{2}(\phi^{-1}F)_{pj,k}-(\phi^{-1})_{qj}\phi_{spq,k}.\]
Thus,
\begin{eqnarray*}8\sqrt{-1}\sigma(\partial_{\bar{z}_n^-})&=&\{(\phi^{-1})_{qj}\phi_{spq,k}(\phi^{-1})_{kl}+(\phi^{-1})_{kj,p}[\delta_{kl}+\frac{\sqrt{-1}}{2}(\phi^{-1}F)_{kl}]\\
&-&\frac{\sqrt{-1}}{2}(\phi^{-1}F)_{kj,p}(\phi^{-1})_{kl}\}\times Q_{jl}\times (\phi^{-1})_{np}.\end{eqnarray*}
On one side, by Eq.~(\ref{Q}) we have
\begin{eqnarray*}(\phi^{-1})_{qj}\phi_{spq,k}(\phi^{-1})_{kl}Q_{jl}&=&[(\phi^T)^{-1}Q\phi^{-1}]_{qk}\phi_{spq,k}\\
&=&2(\phi_s-\frac{\sqrt{-1}}{2}F)^{-1}_{qk}(\phi_s-\frac{\sqrt{-1}}{2}F)_{kq,p}\\
&=&2[\log\det(\phi_s-\frac{\sqrt{-1}}{2}F)]_{,p}\\
&=&[\log\det(\phi_s\Xi)]_{,p},\end{eqnarray*}
where the well-known formula concerning the differential of the determinant of a matrix-valued function $\Theta$
\[d\log\det \Theta=\sum_{i,j}(\Theta^{-1})_{ij}d\Theta_{ji}\]
is used.
On the other side,
\begin{eqnarray*}&-&\frac{\sqrt{-1}}{2}(\phi^{-1}F)_{kj,p}(\phi^{-1})_{kl}Q_{jl}=[\delta_{kj}-\frac{\sqrt{-1}}{2}(\phi^{-1}F)_{kj}]_{,p}[(\phi^T)^{-1}Q]_{jk}\\
&=&2[\delta_{kj}-\frac{\sqrt{-1}}{2}(\phi^{-1}F)_{kj}]_{,p}\phi_{jq}(\phi_s-\frac{\sqrt{-1}}{2}F)^{-1}_{qk}\\
&=&2\{(\phi-\frac{\sqrt{-1}}{2}F)_{kq,p}-[\delta_{kj}-\frac{\sqrt{-1}}{2}(\phi^{-1}F)_{kj}]\phi_{sjq,p}\}\times(\phi_s-\frac{\sqrt{-1}}{2}F)^{-1}_{qk} \\
&=&2\{(\phi_s-\frac{\sqrt{-1}}{2}F)_{kq,p}-[\delta_{kj}-\frac{\sqrt{-1}}{2}(\phi^{-1}F)_{kj}]\phi_{sjq,p}\}\times(\phi_s-\frac{\sqrt{-1}}{2}F)^{-1}_{qk}\\
&=&[\log\det(\phi_s\Xi)]_{,p}-2(\phi^{-1})_{lj}(\phi-\frac{\sqrt{-1}}{2}F)_{kl}\phi_{sjq,p}\times(\phi_s-\frac{\sqrt{-1}}{2}F)^{-1}_{qk}.
\end{eqnarray*}
Since by Eq.~(\ref{Q})\begin{eqnarray*}
[\delta_{kl}+\frac{\sqrt{-1}}{2}(\phi^{-1}F)_{kl}]\times Q_{jl}=2(\phi+\frac{\sqrt{-1}}{2}F)_{kl}(\phi_s-\frac{\sqrt{-1}}{2}F)^{-1}_{ql}\phi_{jq},
\end{eqnarray*}
we have
\begin{eqnarray*}(\phi^{-1})_{kj,p}[\delta_{kl}+\frac{\sqrt{-1}}{2}(\phi^{-1}F)_{kl}] Q_{jl}&=&-2(\phi^{-1})_{kj}(\phi+\frac{\sqrt{-1}}{2}F)_{kl}(\phi_s-\frac{\sqrt{-1}}{2}F)^{-1}_{ql}\phi_{sjq,p}\\
&=&-2(\phi^{-1})_{kj}(\phi^T-\frac{\sqrt{-1}}{2}F)_{lk}(\phi_s-\frac{\sqrt{-1}}{2}F)^{-1}_{ql}\phi_{sjq,p}\\
&=&-2(\phi^{-1})_{lj}(\phi^T-\frac{\sqrt{-1}}{2}F)_{kl}(\phi_s-\frac{\sqrt{-1}}{2}F)^{-1}_{qk}\phi_{sjq,p}.\end{eqnarray*}
Note that
\[(\phi-\frac{\sqrt{-1}}{2}F)+(\phi^T-\frac{\sqrt{-1}}{2}F)=2(\phi_s-\frac{\sqrt{-1}}{2}F),\]
and consequently,
\[-4(\phi^{-1})_{lj}(\phi_s-\frac{\sqrt{-1}}{2}F)_{kl}\phi_{sjq,p}(\phi_s-\frac{\sqrt{-1}}{2}F)^{-1}_{qk}=-4(\phi^{-1})_{lj}\phi_{sjl,p}=-4(\log \det \phi)_{,p}.\]
Combining all the above pieces together, we finally have
\begin{eqnarray*}8\sqrt{-1}\sigma(\partial_{\bar{z}_n^-})&=&2(\phi^{-1})_{np}\{[\log\det(\phi_s\Xi)]_{,p}-2(\log \det \phi)_{,p}\}\\
&=&2(\phi^{-1})_{np}[\log\frac{\det(\phi_s\Xi)}{(\det \phi)^2}]_{,p}.\end{eqnarray*}
Therefore, our $\epsilon$ can be chosen to be
\[\epsilon=[\frac{\det(\phi_s\Xi)}{(\det \phi)^2}]^{-1/2}.\]
Thus we have proved that $s^+=\epsilon s_0^+$ is a $J_-$-holomorphic section of $K_+^{-1}$.
\end{proof}
\emph{Remark}. Similarly, if we set $s_0^-=\partial_{z_1^-}\wedge\cdots\wedge \partial_{z_n^-}$, then $\epsilon s_0^-$ is a $J_+$-holomorphic section of $K_-^{-1}$. These together imply the following proposition:
\begin{proposition}$s=\epsilon s_0^+\otimes s_0^-$ is a GH section of $K_+^{-1}\otimes K_-^{-1}$. In particular,
\[h_\omega(s,s)=c[\det(\phi_s\Xi)]^{-1},\]
where $c$ is a positive constant.
\end{proposition}
\begin{proof}Since $s_0^-$ is obviously $J_-$-holomorphic, this together with Lemma~\ref{holos} shows that $s$ is also $J_-$-holomorphic. Similarly, $s$ is $J_+$-holomorphic.

Obviously, $s_0^+$ and $s_0^-$ have the same length w.r.t. $g$. Let $h_g(s,s)$ be the square of the length of $s$ w.r.t. $g$. Then by Lemma~\ref{det} we immediately have
\[h_g(s,s)=(\frac{1}{2})^{2n}\det \phi_s \times \det \Xi \times (\det \phi)^{-4}\times \epsilon^2=(\frac{1}{2})^{2n}(\det \phi)^{-2},\]
and thus for a positive constant $c_0$ we have\footnote{It should be reminded that we are dealing with $K_+^{-1}\otimes K_-^{-1}$ rather than $K_+^{1/2}\otimes K_-^{1/2}$.}
\[h_\omega(s,s)=c_0(\frac{vol_\omega}{vol_g})^2\times h_g(s,s)=c[\det(\phi^{-1})_s\det \Xi]^{-1}  \times (\det \phi)^{-2}=c[\det(\phi_s\Xi)]^{-1},\]
where we have used the fact that $\phi^T(\phi^{-1})_s\phi=\phi_s$ and $c:=c_0/4^n$. Our conclusion then follows.
\end{proof}
Now according to our analysis in \S~\ref{goto}, \emph{in the metric splitting} and in the local frame $s$, the connection form of generalized connection $D^\omega$ on $K_+^{-1}\otimes K_-^{-1}$ has the following form
\[\chi_0=\frac{1}{2}(\partial_++\partial_-)\log [\det(\phi_s\Xi)]^{-1}+\frac{1}{2}g^{-1}(\partial_+-\partial_-)\log [\det(\phi_s\Xi)]^{-1},\]
while \emph{in the symplectic splitting}, the connection 1-form of the ordinary connection part $\nabla^\omega$ of this generalized connection is thus
\[\chi=\frac{1}{2}(\partial_++\partial_-)\log [\det(\phi_s\Xi)]^{-1}+\frac{1}{2}bg^{-1}(\partial_+-\partial_-)\log [\det(\phi_s\Xi)]^{-1}.\]
\begin{proposition}\label{scal}The curvature of $\nabla^\omega$ on $K_+^{-1}\otimes K_-^{-1}$ is
\[R=\sqrt{-1}d(J_+^*+J_-^*)^{-1}d\log [\det(\phi_s\Xi)]^{-1},\]
and thus the scalar curvature $\kappa$ has the following form
\begin{equation}\kappa=\partial_i[(\partial_j\log \det(\phi_s\Xi)^{1/2})(\Xi^{-1})_{ij}],\label{sca}\end{equation}
where $\partial_i=\partial_{\mu_i}$.
\end{proposition}
\begin{proof}If $f$ is a smooth function depending only on $\mu$, then
\[(\partial_++\partial_-)df=df-\frac{\sqrt{-1}}{2}(J_+^*+J_-^*)df\]
and
\[(\partial_+-\partial_-)df=-\frac{\sqrt{-1}}{2}(J_+^*-J_-^*)df.\]
Note that \[bg^{-1}=\omega(J_+-J_-)(J_++J_-)^{-1}\omega^{-1}=-(J_+^*-J_-^*)(J_+^*+J_-^*)^{-1}.\] Therefore, we have
\begin{eqnarray*}bg^{-1}(\partial_+-\partial_-)df&=&\frac{\sqrt{-1}}{2}(J_+^*-J_-^*)(J_+^*+J_-^*)^{-1}(J_+^*-J_-^*)df\\
&=&-\frac{\sqrt{-1}}{2}(J_+^*-J_-^*)^2(J_+^*+J_-^*)^{-1}df,\end{eqnarray*}
and consequently
\begin{eqnarray*}&-&\frac{\sqrt{-1}}{2}(J_+^*+J_-^*)df+bg^{-1}(\partial_+-\partial_-)df\\
&=&-\frac{\sqrt{-1}}{2}[(J_+^*+J_-^*)^2+(J_+^*-J_-^*)^2](J_+^*+J_-^*)^{-1}df\\
&=&2\sqrt{-1}(J_+^*+J_-^*)^{-1}df.\end{eqnarray*}
Therefore, the curvature of $\nabla^\omega$ is
\[R=d\chi=\sqrt{-1}d(J_+^*+J_-^*)^{-1}d\log [\det(\phi_s\Xi)]^{-1}.\]
As for the scalar curvature, we note that the Ricci form is
\[P_1=d(J_+^*+J_-^*)^{-1}d\log [\det(\phi_s\Xi)]^{1/2}\]
and that
\begin{equation}(\frac{J_+^*+J_-^*}{2})^{-1}d\mu=\Xi^{-1}d\theta-\frac{1}{2}\Xi^{-1}F(\phi_s)^{-1}\phi_ad\mu.\label{+-}\end{equation}
Substitute these into Eq.~(\ref{sc}) and note that on the right side of Eq.~(\ref{+-}) only the first term has contribution to the scalar curvature. An elementary computation leads to the final expression of $\kappa$.
\end{proof}
\emph{Remark}. Though $\phi_a$ has no contribution to $\kappa$, it does to the Ricci curvature. If $F=0$, then we have
\[\kappa=\partial_i[(\partial_j\log \det\phi_s)(\phi_s)^{-1}_{ij}],\]
which coincides with $\kappa=-(\phi_s)_{ij,ij}$. This equivalence is classical and can be found in \cite{Ab}. However, if $F\neq 0$ the equivalence of Eq.~(\ref{sca}) with Eq.~(\ref{sca1}) is not evident. This will be addressed in the next subsection.

\subsection{Equivalence between the two versions of scalar curvature}\label{equ}
From the viewpoint of scalar curvature as a moment map which was carried out by Boulanger, we know that for a toric GK manifold of symplectic type, the scalar curvature is of the following form
\[\kappa=-(\Xi^{-1})_{ij,ij}\]
where $\Xi=\phi_s+\frac{1}{4}F(\phi_s)^{-1}F$. On the other side, from the viewpoint of our refined version of Goto's formalism, we note that $\kappa$ takes another seemingly rather different form
\[\kappa=\partial_i[(\partial_j\log \det(\phi_s\Xi)^{1/2})(\Xi^{-1})_{ij}].\]
The goal of this subsection is then to show that the two expressions are actually the same.
\begin{proposition}For toric GK manifolds of symplectic type, Boulanger's scalar curvature is the same as Goto's, i.e.
\[-(\Xi^{-1})_{ij,ij}=\partial_i[(\partial_j\log \det(\phi_s\Xi)^{1/2})(\Xi^{-1})_{ij}].\]
\end{proposition}
\begin{proof}

First, recall that
\[\det(\phi_s+\frac{\sqrt{-1}}{2}F)=\det (\phi_s\Xi)^{1/2}.\]
 Additionally, $M_1^{-1}$ in the proof of Lemma~\ref{det}, of course, also commutes with $J$ there. Then the restriction of $M_1^{-1}$ on $\mathbb{C}^n$ implies the following formula which will be used later:
\begin{equation}\Xi^{-1}=(\phi_s-\frac{\sqrt{-1}}{2}F)^{-1}\phi_s(\phi_s+\frac{\sqrt{-1}}{2}F)^{-1}.\label{FI}\end{equation}
Now using the formula concerning the differential of the determinant of a matrix-valued function $\Theta$
\[d\log\det \Theta=\sum_{i,j}(\Theta^{-1})_{ij}d\Theta_{ji}\]
again, we have
\[
(\partial_j\log \det(\phi_s\Xi)^{1/2})(\Xi^{-1})_{ij}=(\phi_s+\frac{\sqrt{-1}}{2}F)^{-1}_{pq}\phi_{sqp,j}(\phi_s-\frac{\sqrt{-1}}{2}F)^{-1}_{mi}\phi_{snm}(\phi_s+\frac{\sqrt{-1}}{2}F)^{-1}_{jn}.\]
Note that
\[(\phi_s+\frac{\sqrt{-1}}{2}F)^{-1}_{pq}\phi_{sqp,j}=(\phi_s+\frac{\sqrt{-1}}{2}F)^{-1}_{pq}\phi_{sqj,p}=-(\phi_s+\frac{\sqrt{-1}}{2}F)^{-1}_{pq,p}(\phi_s+\frac{\sqrt{-1}}{2}F)_{qj}\]
where we have used the fact that $\phi_s$ is the Hessian of $\tau$. Now we have
\begin{eqnarray*}(\partial_j\log \det(\phi_s\Xi)^{1/2})(\Xi^{-1})_{ij}&=&-(\phi_s+\frac{\sqrt{-1}}{2}F)^{-1}_{pq,p}\phi_{sqm}(\phi_s-\frac{\sqrt{-1}}{2}F)^{-1}_{mi}\\
&=&-\Xi^{-1}_{pi,p}+(\phi_s+\frac{\sqrt{-1}}{2}F)^{-1}_{pq}[\phi_{sqm}(\phi_s-\frac{\sqrt{-1}}{2}F)^{-1}_{mi}]_{,p}\end{eqnarray*}
where Eq.~(\ref{FI}) is used. If we could further prove that the second term of the last line vanishes, then we would be done. Denote this term by $q_i$. Then
\begin{eqnarray*}q_i&=&(\phi_s+\frac{\sqrt{-1}}{2}F)^{-1}_{pq}[\delta_{iq}+\frac{\sqrt{-1}}{2}F_{qm}(\phi_s-\frac{\sqrt{-1}}{2}F)^{-1}_{mi}]_{,p}\\
&=&-\frac{\sqrt{-1}}{2}(\phi_s+\frac{\sqrt{-1}}{2}F)^{-1}_{pq}F_{qm}(\phi_s-\frac{\sqrt{-1}}{2}F)^{-1}_{ni}\phi_{srn,p}(\phi_s-\frac{\sqrt{-1}}{2}F)^{-1}_{mr}.\end{eqnarray*}
It can be easily checked that the coefficient of $F_{qm}$ is symmetric in $q, m$ (that $\phi_s$ is a Hessian should be used here again). Since $F_{qm}=-F_{mq}$, $q_i$ consequently really vanishes! This completes our proof.
\end{proof}
\section*{Acknowledgemencts}
This study is supported by the Natural Science Foundation of Jiangsu Province (BK20150797). The manuscript is prepared during the author's stay in the Department of Mathematics at the University of Toronto and this stay is funded by the China Scholarship Council (201806715027). The author also thanks Professor Marco Gualtieri for his invitation and hospitality.

\end{document}